\documentclass[12pt]{article}
\usepackage{a4wide}

\usepackage{latexsym,amssymb,amsmath,amsthm,epsfig,amsfonts,algorithm}

\usepackage{eucal}
\usepackage{multirow,multicol,array,bm}
\usepackage[authoryear]{natbib}
\usepackage{tabularx, booktabs, makecell, caption}
\usepackage{siunitx}
\usepackage[colorlinks=true,allcolors=blue]{hyperref}
\usepackage{float}
\usepackage{caption}
\usepackage[utf8]{inputenc}
\usepackage[english]{babel}
\usepackage[noend]{algpseudocode}
\usepackage{tikz}
\usetikzlibrary{arrows}
\usepackage{bbm}
\usepackage[noend]{algpseudocode}
\usepackage{graphicx}
\usepackage[skip=0.333\baselineskip]{subcaption}
\usepackage[font=normal,labelfont=bf]{caption}
\usepackage{algpseudocode}
\usepackage{mathtools}
\usepackage{mathrsfs}  
\usepackage{bbm}

\algdef{SE}[DOWHILE]{Do}{doWhile}{\algorithmicdo}[1]{\algorithmicwhile\ #1}%
\newtheorem{lemma}{Lemma}
\newtheorem{paradox}{Paradox}
\newtheorem{corollary}{Corollary}
\newtheorem{theorem}{Theorem}
\newtheorem{definition}{Definition}

\numberwithin{table}{section}
\numberwithin{equation}{section}

\begin{document}

\title{Strategic Customer Behavior in an M/M/1 Feedback Queue with General Payoffs}

\author{
  Peter Taylor\thanks{School of Mathematics and Statistics, The University of Melbourne, Australia. \texttt{taylorpg@unimelb.edu.au}}
  \and
  Jiesen Wang\thanks{Korteweg-de Vries Institute, University of Amsterdam, Netherlands. \texttt{j.wang2@uva.nl}}
}

\maketitle
\begin{abstract}


We consider an M/M/1 feedback queue in which service attempts may fail, requiring the customer to rejoin the queue. Arriving customers act strategically, deciding whether to join the queue based on a threshold strategy that depends on the number of customers present. Their decisions balance the expected service reward against the costs associated with waiting, while accounting for the behavior of others. 

This model was first analyzed by Fackrell, Taylor and Wang (2021), who assumed that waiting costs were a linear function of the time in the system. They showed that increasing the reward for successful service or allowing reneging can paradoxically make all customers worse off. In this paper, we adopt a different setting in which waiting does not incur direct costs, but service rewards are subject to discounting over time. We show that under this assumption, paradoxical effects can still arise. 

Furthermore, we develop a numerical method to recover the sojourn time distribution under a threshold strategy and demonstrate how this can be used to derive equilibrium strategies under other payoff metrics.



\end{abstract}

\section{Introduction}

In the study of strategic behavior in queueing systems, the waiting time is typically incorporated into the payoff function as a cost. Much of the literature assumes this cost to be linear in the waiting time; see, for example, \citet{N69}, \citet{HH95}, \citet{HR15}, \citet{HO18}, and \citet{FTW21}. Under this assumption, evaluating the expected payoff reduces to computing the expected waiting time.

However, many practical problems involve nonlinear waiting costs that depend on the entire distribution of waiting times, not merely on the first moment. In \citet{HR01}, the authors analyzed an \( M/M/m \) queue with nonlinear waiting costs, allowing customers to renege, and derived the homogeneous Nash equilibrium distribution of abandonment times. Their analysis evaluated the waiting time distribution of a tagged customer—who joins the queue without reneging while all others abandon after a random time—to compute the customer’s expected future gain. The authors of \citet{HR01} observed that assuming a convex cost function is common when modeling reneging behavior, as it reflects increasingly adverse conditions associated with longer waiting times. Similarly, \citet{SM04} investigated reneging behavior by incorporating customers’ subjective beliefs about their waiting time distributions into the determination of expected payoffs. In contrast, \citet{HH95} assumed a linear waiting cost but modeled the service value as constant for a fixed time after arrival, dropping to zero thereafter, thereby capturing the influence of the tail of the waiting time distribution.
	
In economics, future costs and rewards are typically assumed to be continuously discounted at a positive interest rate. As a result, discounted rewards are commonly used in payoff formulations; see, for example, \citet{S72}, \citet{2H75}, \citet{CF01}, and \citet{LMT22}. When discounted rewards are considered, it is standard to include an outside option or opportunity—see \citet{CF01} for a representative example. From a mathematical standpoint, this assumption is also essential: without an outside opportunity, it may become optimal to wait indefinitely for a reward that is so heavily discounted over time that it effectively becomes negligible. 
	
From a modelling perspective, we argue that a linear cost measure is likely to be more appropriate when the sojourn time is short—such as waiting for a haircut or a cup of coffee. However, when the sojourn time is long, a discounted reward may be a more suitable measure. For example, in a kidney transplant waiting system \citep{SZ04}, a patient's health or the organ quality may deteriorate over time, making it reasonable to assume that the reward is continuously discounted at some rate.

In this paper, we revisit the model introduced by \citet{FTW21}, but adopt a discounted reward structure for the payoff. Specifically, we consider an M/M/1 queueing system, where each customer pays a cost \( v\) to join the queue. The customer remains in the queue until it receives service. However there is a probability \( 1-q \) that the service fails, in which case the customer returns to the end of the queue. When the customer eventually receives a successful service, she is allocated a reward \( e^{-\alpha W} \) where $W$ is the time that she has been in the queue. 

The service discipline follows a first-come-first-served (FCFS) policy. An arriving customer decides whether to join the system based on whether doing so yields a positive expected payoff. We assume the system begins empty, with no customers initially present.
 
We first assume that joining customers remain in the system until they successfully complete the service and establish the Nash equilibrium threshold strategy. Then we analyze the case where customers are allowed to leave every time their service fails. 
We abbreviate the above two situations as the $N$-case (non-reneging) and the $R$-case (reneging), respectively. 


We derive the symmetric Nash equilibrium thresholds in both the \( N \)-case and the \( R \)-case. Our results show, counterintuitively, that the stationary individual payoff may increase when the discount rate \( \alpha \) or the cost of the outside option \( v \) increases. Furthermore, allowing reneging does not necessarily make customers better off, consistent with the observation in \citet[Paradox 2]{FTW21}.

In the above analysis, the discount parameter $\alpha$ is treated as a fixed parameter that models how the value of the reward decreases over time. However, there are two alternative interpretations for this parameter. 

First, $\alpha$ can be thought of as the rate of an independent exponential {\it killing variable} $\tau$, which terminates the operation of the whole process. Denoting the waiting time distribution of the tagged customer by \( W(u)\), \(\mathbb{E}[e^{-\alpha W}]   = \int_0^\infty e^{-\alpha u} dW(u) \) is the expected value of \(I(W < \tau)\) the indicator random variable that the soujourn time finishes before the process terminates. The mathematical analysis that arises from such a model is identical to that arising in the discounted case. However, physically, it models a system where the reward for successful service stays fixed, but it is earned only if the process has not been killed when the sojourn time is completed. We shall use this interpretation in  Sections~\ref{sec:ExNRcase} and \ref{sec:ExRcase}, where we use {\it coupling} to derive some stochastic monotonicity results.

Second, $\alpha$ can be thought of as the variable in the Laplace transform 
\(\mathbb{E}[e^{-\alpha W}]\) of the sojourn time. We shall use this interpretation in 
Section~\ref{sec:ExTimeDist}, to compute the sojourn time distribution by applying numerical inversion and proceed to evaluate equilibrium strategies under other payoff metrics. 

Thus, $\alpha$ plays three distinct roles in our analysis:
\begin{enumerate}\label{threeroles}
    \item as a discount rate,
    \item as an exponential killing rate, and
    \item as a Laplace transform variable.
\end{enumerate} 

This paper is organised as follows. Section~\ref{sec:preliminary} introduces the model and outlines the key components needed for our analysis. In Section~\ref{sec:ExNRcase}, we focus on the $N$-case and use matrix-analytic methods to compute the expected discounted reward of a tagged customer, conditioned on her joining position and the threshold strategy used by others. We then determine a symmetric Nash equilibrium threshold.
Section~\ref{sec:ExRcase} extends the model by allowing customers to renege whenever their service fails, assuming that the reneging and joining strategies follow the same threshold. We compute the corresponding Nash equilibrium and compare it with that of the $N$-case.
In Section~\ref{sec:ExParadox}, we present two paradoxes. Section~\ref{sec:ExTimeDist} numerically computes the sojourn time distribution of a tagged customer, given her joining position and the threshold used by others. We illustrate how this distribution can be applied to evaluate alternative payoff structures through an example.

\section{Background and Preliminaries}
\label{sec:preliminary}

Our work analyzes how customer behavior influences system dynamics, which in turn affects individual decision-making. In this section, we introduce the model under study; the joining strategy adopted by customers, assuming a common strategy across all individuals; the definition of a Nash equilibrium, which characterizes the strategic profile; and the mathematical tools used in the analysis.

We consider an M/M/1 feedback queue with arrival rate \(\lambda\) and service rate \(\mu\), operating under a FCFS discipline. After a service is completed, with probability \(q\), the service is deemed successful and the customer leaves the system or, with probability \(1 - q\), the service is deemed a failure and the customer rejoins the end of the queue.

Since service times are exponentially distributed, the sojourn time of an arriving customer depends on the number of customers already in the system, given the strategies of others. Therefore, it is reasonable to assume that customers adopt a threshold strategy: they join the system only if the number of customers present does not exceed a certain threshold.  
We now formally define a generalized threshold strategy with a real-valued parameter, originally introduced by \citet{H96}. Such a threshold strategy is a function \( u: \mathbb{Z} \rightarrow [0,1] \) that determines the joining behavior of an arriving customer. Specifically, \( u(i) \) represents the probability that a customer will join the queue upon finding \( i-1 \) customers already present (thereby taking position \( i \) if she joins). 

\begin{definition}[Threshold Strategy] \label{D1}  
A threshold strategy \(\mathcal{T}(x)\), where \(x \in \mathbb{R}^+\) and \(n \equiv \lfloor x \rfloor\), is characterized by the joining probability function
\begin{equation}
    u^{(x)}(i) = 
    \begin{cases}
        1 & \text{if } i \leq n, \\
        p \equiv x - n & \text{if } i = n + 1, \\
        0 & \text{if } i \geq n + 2,
    \end{cases}
\end{equation}
where \(i\) denotes the position upon joining.
\end{definition}  
\noindent A customer following threshold strategy \(\mathcal{T}(x)\) always joins the queue if her position would be at most \( x \), that is, \( i \leq \lfloor x \rfloor \). If her position would be \( \lfloor x \rfloor + 1 \), she joins with probability \( p = x - \lfloor x \rfloor \). Otherwise for positions \( i \geq \lfloor x \rfloor + 2 \), she does not join.

In \citet{FTW21}, the payoff was defined as the difference between a fixed reward earned upon successful service completion and the expected sojourn time. Here, we generalize this by assuming that a unit initial reward is discounted continuously at rate \(\alpha \geq 0\) over the customer’s sojourn time \(W\), resulting in a net reward of \(e^{-\alpha W}\).
Customers must also pay an admission fee \(v\) upon joining the system. Each arriving customer observes the current queue length and decides whether to join or balk based on their expected payoff,
\begin{equation} \label{eq:payoff1}
    \mathbb{E}[e^{-\alpha W}] - v.
\end{equation}
Customers join if this expected payoff is positive, balk if it is negative, and are indifferent when it equals zero. 

We assume all arriving customers adopt identical threshold strategies. A strategy \(\mathcal{T}(x)\) forms a Nash equilibrium when no customer can benefit by unilaterally deviating from it. The best response depends on the chosen payoff criterion. To determine the Nash equilibrium, we arbitrarily select a customer, referred to as the tagged customer, and evaluate whether \(\mathcal{T}(x)\) is optimal for her, assuming all others follow it as well.

The stationary distribution of the queue length is determined by customers' threshold choices. Having derived the Nash equilibrium strategy, we proceed to compare the expected payoffs when customers adopt the equilibrium threshold strategy under different payoff structures in the stationary regime. By the PASTA (Poisson Arrivals See Time Averages) property, if the system is in a stationary state, then the number of customers already present when a customer arrives follows the stationary distribution $\mathbb{\pi} =(\pi_i^{(x)})$ and so the stationary expected individual payoff is
\begin{equation} \label{eq:payoff}
    V(x) := \sum_{i=0}^{\lfloor x \rfloor} \pi_i^{(x)} V_i^{(x)},
\end{equation}
where $V_i^{(x)}$ represents the expected payoff for a customer who joins when the system is in this state. 

\section{The Case Where Customers Cannot Renege}  \label{sec:ExNRcase}

In this section, we analyze customers' decisions based on whether the expected payoff in \eqref{eq:payoff1} is positive. To compute the symmetric Nash equilibrium, we arbitrarily select a customer as the tagged customer and calculate her expected payoff, assuming that all other customers adopt a common threshold strategy \(\mathcal{T}(x)\), while the tagged customer may use a different threshold. We begin with a Markovian model that enables us to derive the distribution of the sojourn time.

For each arrival state $1 \leq i \leq \lfloor x \rfloor + 1$, let $W_i^{(x)}$ denote the sojourn time of a tagged customer who joins the system in position $i$, assuming all other customers follow a threshold strategy $x$. To compute $W_i^{(x)}$, we analyze a Markov jump process that tracks the progress of the tagged customer through the system. Let
\begin{itemize}
\item $i$ be the current position of the tagged customer in the queue, and 
\item $j$ be the total number of customers in the system.
\end{itemize}
For $1 \leq i \leq j \leq \lfloor x \rfloor + 1$, let $W_{i,j}^{(x)}$ denote the time until the tagged customer completes service and departs the system, given that there are initially $j$ customers in the system, she is in position $i$, and all other customers use threshold $x$. Note that $W_i^{(x)}$ has the same distribution as $W_{i,i}^{(x)}$.

The $W_{i,j}^{(x)}$ are times to absorption for a continuous-time Markov chain $\{X(t), t \geq 0\}$, with state space 
\begin{equation}
\label{eq:state}
\mathcal{S} \equiv \left\{(i, j) \,:\, 1 \leq i \leq j \leq \lfloor x \rfloor + 1 \right\} \cup \{0\},
\end{equation}
with state $0$ representing the situation in which the tagged customer has exited the system. This process is a continuous-time quasi-birth-and-death (QBD) process on $\mathcal{S}$, with $j$ representing the level and $i$ the phase. All states $(i, j)$ are transient, while state $0$ is absorbing.

Let $T(t)$ be the time until the first transition after time $t$, then $T(t) - t$ is exponentially distributed with rate $\lambda+\mu$. At time $T(t)$, an arrival occurs with probability $ {\lambda}/{(\lambda+\mu)}$. When $j < \lfloor x \rfloor$, the arriving customer joins the system with probability $1$; when $j = \lfloor x \rfloor$, the arriving customer joins the queue with probability $p$; when $j = \lfloor x \rfloor+1$, the arriving customer balks. 
On the other hand, a service completion occurs at time $T(t)$ with probability ${\mu}/{(\lambda+\mu)}$, after which a customer leaves the system with probability $q$ and rejoins the end of the queue with probability $1-q$. If the customer in service is the tagged customer, which occurs when $i=1$, her future sojourn time is $0$ with probability $q$. Otherwise, her next position is $j$. If the customer in service is not the tagged customer, the position of the tagged customer decreases by $1$, and the total number of customers decreases by $1$ with probability $q$ and stays unchanged with probability $1-q$.

Therefore, for $1 \leq i \leq j \leq \lfloor x \rfloor + 1$ and $t \geq 0$, given $X(t) = (i,j)$, the conditional transition probabilities are:
{\scriptsize	\begin{align} \label{eq:transition1}
	P\{X(T(t)) = (i',j') \mid X(t) = (i,j)\} & \\
	=  & \begin{dcases}
	\frac{\lambda}{\lambda+\mu} & \text{if} \quad  j < \lfloor x \rfloor, \, (i',j') = (i,j+1) \text{ or }   j > \lfloor x \rfloor, \, (i',j') = (i,j) \\
	\frac{\lambda p}{\lambda+\mu} & \text{if} \quad j = \lfloor x \rfloor, \, (i',j') = (i,j+1) \\
	\frac{\lambda (1-p)}{\lambda+\mu} & \text{if} \quad  j = \lfloor x \rfloor, \, (i',j') = (i,j) \\
	\frac{\mu q}{\lambda+\mu} & \text{if} \quad   i =1, \, (i',j') = 0 \text{ or }  i >1, \, (i',j') = (i-1,j-1)\\
	\frac{\mu (1-q)}{\lambda+\mu} & \text{if} \quad  i =1, \, (i',j') = (j,j)  \text{ or }  i >1, \, (i',j') = (i-1,j)  \,. 
	\end{dcases} \notag
\end{align}}
Assuming that the state is $(i,j)$ at time $t$, the expected discounted reward $\mathbb{E}\left( e^{-\alpha W^{(x)}_{i,j}}\right)$ is given by

{\scriptsize
\begin{align} \label{eq:Wij}
	&\mathbb{E}\left(e^{-\alpha W^{(x)}_{i,j}}\right) \\
	=& \mathbb{E}\left( \mathbb{E}\left(e^{-\alpha W^{(x)}_{i,j}} \mid T(t), X(T(t))\right) \right) \notag \\
	=& \int_{0}^{\infty} (\lambda+\mu) \, e^{-(\lambda+\mu) y}  \, \sum_{(i', j')} P(X(T(t)) = (i',j') \mid X(t) = (i,j) ) \, \mathbb{E}\left(e^{-\alpha W^{(x)}_{i,j}} \mid T(t) = y, X(T(t)) = (i',j') \right)  dy \notag \\
	=& \int_{0}^{\infty} (\lambda+\mu) \, e^{-(\lambda+\mu) y}  \, \sum_{(i', j')} P(X(T(t)) = (i',j') \mid X(t) = (i,j) ) \, \mathbb{E}\left(e^{-\alpha \left(y + W^{(x)}_{i',j'} \right)} \right)  dy \notag\\
	=& \left( \sum_{(i', j')} P(X(T(t)) = (i',j') \mid X(t) = (i,j) ) \, \mathbb{E} \left(e^{-\alpha W^{(x)}_{i',j'}} \right) \right) \, \int_{0}^{\infty} (\lambda+\mu) \, e^{-(\lambda+\mu+\alpha) y}  \,  dy \notag \\
	=& \left( \frac{\lambda}{\lambda+\mu} \left( \mathbb{E} \left(e^{-\alpha W^{(x)}_{i, j+1}}\right)\mathbbm{1}_{\{i <  \lfloor x \rfloor\}}  + \left( p \,  \mathbb{E} \left(e^{-\alpha W^{(x)}_{i, j+1}}\right) + (1-p) \, \mathbb{E} \left(e^{-\alpha W^{(x)}_{i, j}}\right) \right) \mathbbm{1}_{\{i = \lfloor x \rfloor\}}  + \mathbb{E} \left(e^{-\alpha W^{(x)}_{i, j}}\right) \mathbbm{1}_{\{i >  \lfloor x \rfloor\}}  \right) \right.\notag \\  
	&  \left. + \frac{\mu}{\lambda+\mu} \, \left( \left( q \,\mathbb{E} \left(e^{-\alpha W^{(x)}_{i-1, j-1}}\right) + (1-q) \, \mathbb{E} \left(e^{-\alpha W^{(x)}_{i-1, j}}\right) \right) \mathbbm{1}_{\{i > 1\}} + \left(q  + (1-q) \, \mathbb{E} \left(e^{-\alpha W^{(x)}_{j, j}}\right) \right) \mathbbm{1}_{\{i = 1\}}  \right) \right) \frac{\lambda+\mu}{\lambda+\mu+\alpha}  \notag\\
	=& \frac{\lambda}{\lambda+\mu+\alpha} \left( \mathbb{E} \left(e^{-\alpha W^{(x)}_{i, j+1}}\right)\mathbbm{1}_{\{i <  \lfloor x \rfloor\}}  + \left( p \,  \mathbb{E} \left(e^{-\alpha W^{(x)}_{i, j+1}}\right) + (1-p) \, \mathbb{E} \left(e^{-\alpha W^{(x)}_{i, j}}\right) \right) \mathbbm{1}_{\{i =  \lfloor x \rfloor\}}  + \mathbb{E} \left(e^{-\alpha W^{(x)}_{i, j}}\right) \mathbbm{1}_{\{i >  \lfloor x \rfloor\}}  \right) \notag \\
	 &  + \frac{\mu}{\lambda+\mu+\alpha} \, \left( \left( q \,\mathbb{E} \left(e^{-\alpha W^{(x)}_{i-1, j-1}}\right) + (1-q) \, \mathbb{E} \left(e^{-\alpha W^{(x)}_{i-1, j}}\right) \right) \mathbbm{1}_{\{i > 1\}} + (1-q) \, \mathbb{E} \left(e^{-\alpha W^{(x)}_{j, j}}\right)\mathbbm{1}_{\{i = 1\}}  \right) \notag \\
     &+ \frac{\mu q}{\lambda+\mu+\alpha} \, \mathbbm{1}_{\{i = 1\}} \,.\notag
\end{align}}
The final line models the completion of the tagged customer's sojourn in the system, while the remaining lines model transitions to other transient states. If we define
\begin{equation}\label{eq:totalDisFactor1} 
\bm{\gamma}^{(x)}_\alpha \equiv \left[ \mathbb{E}\left(\, e^{-\alpha W^{(x)}_{1,1}}\right) ,  \mathbb{E}\left(\, e^{-\alpha W^{(x)}_{1,2}}\right),  \mathbb{E}\left(\, e^{-\alpha W^{(x)}_{2,2}}\right), \ldots,  \mathbb{E}\left(\, e^{-\alpha W^{(x)}_{\lfloor x \rfloor+1,\lfloor x \rfloor+1}}\right) \right]^T,
\end{equation}
then $\bm{\gamma}^{(x)}_\alpha$ satisfies the linear system
\begin{equation}
\label{eq:EPoisson}
\left(I - P_\alpha^{(x)}\right) \bm{\gamma}^{(x)}_\alpha = \bm{g}_\alpha \,,
\end{equation}
where
\begin{equation} \label{eq:g}
\bm{g}_{\alpha} = \frac{\mu q}{\alpha+\lambda+\mu} \sum_{j = 1}^{\lfloor x\rfloor +1}  \, \bm{e}_{\frac{j(j-1)}{2}+1} \,,
\end{equation}
and $P_{\alpha}^{(x)}$ is defined in Appendix \ref{appendix:1.1}. We can rewrite \eqref{eq:g} in the form 
\begin{equation} \label{eq:d}
\bm{\gamma}^{(x)}_\alpha =  \frac{\mu q}{\alpha+\lambda+\mu} \, \left(I-P_{\alpha}^{(x)}\right)^{-1} \, \bm{d} \,,
\end{equation}
where the vector \( \bm{d} \equiv \sum_{j = 1}^{\lfloor x\rfloor +1}  \, \bm{e}_{\frac{j(j-1)}{2}+1}\) has a one in the components that correspond to the states where the tagged customer is being served (those that have the form \( (1, \ell) \) for \( 1 \leq \ell \leq \lfloor x \rfloor + 1 \)) and zeros elsewhere. The \( (i, j) \)th component of the vector \( \big(I - P_{\alpha}^{(x)} \big)^{-1} \bm{d}\) gives the expected total number of visits to these states before the process is killed or absorbed, conditional on the initial state being \( (i, j) \).

For any $\alpha \geq 0$ and threshold $x$, the matrix $P_{\alpha}^{(x)}$ is irreducible and has at least one row whose sum is strictly less than 1. By a corollary of the Perron–Frobenius theorem \cite[page 8]{S06}, all eigenvalues $r^{(x)}$ of $P_{\alpha}^{(x)}$ satisfy $|r^{(x)}| < 1$. Consequently, any real eigenvalue $1 - r^{(x)}$ of the matrix $I - P_{\alpha}^{(x)}$ is strictly positive. This implies that $\det(I - P_{\alpha}^{(x)}) \neq 0$, and therefore $I - P_{\alpha}^{(x)}$ is nonsingular.

Equation \eqref{eq:EPoisson} is a special case of Poisson’s equation for a level-dependent QBD process. The QBD structure allows us to use the methodology developed in \citet[Theorem 3.3]{DLL13} to compute $\bm{\gamma}^{(x)}_\alpha$ efficiently. Specifically, we solve \eqref{eq:EPoisson} using Algorithm \ref{alg:Poisson} in Appendix \ref{appendix:algorithm1}. For a detailed explanation of the matrices $U^{(j)}, \Gamma^{(j)}, G^{(j)}$ used in Algorithm \ref{alg:Poisson}, see \citet[Chapter 12]{LR99}. Note that the matrix $R^{(j)}$ in \citet[Chapter 12]{LR99} corresponds to our matrix $\Gamma^{(j)}$.

\begin{figure}[h!]
	\centering
	\includegraphics[width = 0.79\textwidth]{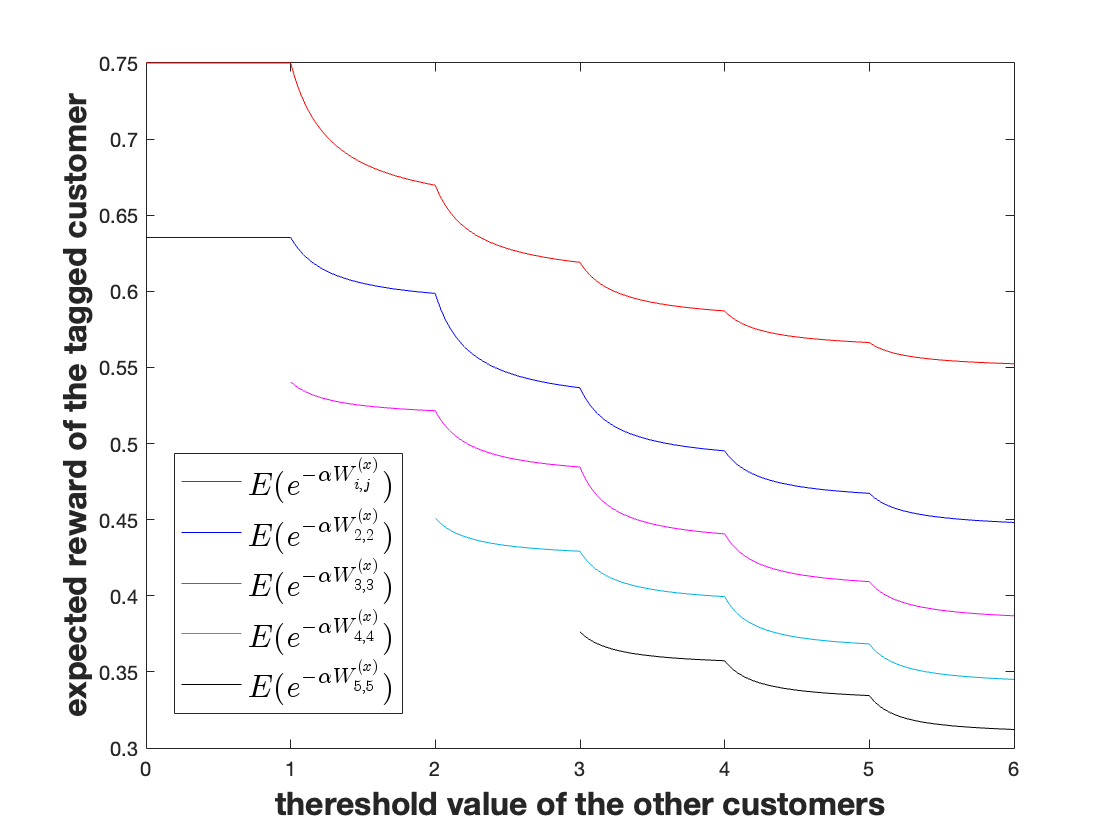}
	\caption{Expected discounted reward for the tagged customer, when $\lambda = 1, \mu = 0.5, q = 0.3, \alpha = 0.05$.} \centering \label{fig:ERewardN}
\end{figure}

In Figure \ref{fig:ERewardN}, we have plotted $\mathbb{E}\left(e^{-\alpha W_{i,i}^{(x)}}\right)$ for $0 \leq x \leq 6$ and $1 \leq i \leq 5$ . The function $\mathbb{E}\left(e^{-\alpha W_{i,i}^{(x)}}\right)$ is meaningful only when $\lfloor x \rfloor \geq \max\{i - 2,0\}$, since the tagged customer cannot join in a position greater than $\lfloor x \rfloor + 2$. Moreover, it remains constant for $0 \leq x \leq 1$ because no other customers will join the system while the tagged customer is present if their threshold $x$ is less than $1$. In this regime, the value of $x$ has no impact on the tagged customer’s expected reward.
When $x$ is small, we can compute $\bm{\gamma}_\alpha^{(x)}$ analytically by solving Equation~\eqref{eq:EPoisson}. For example, for $0 \leq x \leq 1$ and $n = 1,2$, we obtain
\begin{equation}
\mathbb{E}\left(e^{-\alpha W_{1,1}^{(x)}}\right)= \frac{\mu q}{\alpha+\mu q} \qquad \mathbb{E}\left(e^{-\alpha W_{2,2}^{(x)}}\right) = \frac{\mu^2 q (\alpha+2\mu q-\mu q^2)}{(\alpha+\mu q)(\alpha^2+2\mu \alpha +2 \mu^2 q-\mu^2 q^2)}  \,.
\end{equation}

We now view the system from the second point of view in Section \ref{threeroles}. That is, we think of $\alpha$ as the rate of an exponential random variable that kills the process independently of its current state. 
The state space of this continuous-time Markov chain is still \( \mathcal{S}\) as defined in (\ref{eq:state}), and the transition rates associated with arrivals and services are the same. However, there is an extra transition rate \( q((i,j),0) = \alpha\) for every state \((i,j)\). 

The matrix \( P_{\alpha}^{(x)} \) can thus be interpreted as the transition matrix of the discrete-time jump chain constructed for the killed chain in a similar way that the transition matrix with entries given in (\ref{eq:transition1}) was constructed for the chain without killing. This discrete-time chain has a substochastic transition matrix. It can be absorbed either when the tagged customer departs the system or when the process is killed. 

The transition structure reflected in \( P_{\alpha}^{(x)} \) is the same as in \eqref{eq:transition1}, except that the denominator in the transition probabilities is replaced by \( \lambda + \mu + \alpha \). The transition probabilities in \eqref{eq:transition1} constitute the $\alpha = 0$ case. 

When \( \alpha > 0 \), the row sums of the matrix \( P_{\alpha}^{(x)} \) are strictly less than one, reflecting the fact that the process can enter the absorbing state from any transient state. Specifically, transitions to the absorbing state occur from state $(i,j)$ ($i \neq 1$) with probability 
$\alpha/(\lambda + \mu + \alpha)$, and from state $(1,j)$ with probability 
$(\alpha + \mu q)/(\lambda + \mu + \alpha)$.

For any value of $\alpha \geq 0$, let \( N^{(x)}_{i,j} \) to be the number of visits to states of the form \( (1, j) \) before absorption, given that the current state is \( (i,j) \) and all other customers follow the threshold strategy \( x \). The following two lemmas concern the stochastic behavior of \( N^{(x)}_{i,j} \).

\begin{lemma} \label{lemma:1}
    For \( i_1 < i_2 \), the random variable \( N^{(x)}_{i_1,i_1} \) stochastically dominates \( N^{(x)}_{i_2,i_2} \).
\end{lemma}
\begin{proof}
We prove this by {\it coupling}. Consider two positions $i_1$ and $i_2$ with $i_1 < i_2$. Suppose there are two queues, $Q_1$ and $Q_2$, where a tagged customer joins at position $i_1$ in $Q_1$ and at position $i_2$ in $Q_2$. We first allow $Q_2$ to evolve while {\it freezing} $Q_1$ until either the $Q_2$ process is absorbed (which must necessarily occur as a result of killing) or the tagged customer reaches position $i_1$. If the latter occurs, the number of customers in $Q_2$ must be at least as large as in $Q_1$.

Next, we {\it unfreeze} $Q_1$ and couple the two systems so that they evolve on the same probability space: all arrival times, service durations, the killing random variable and Bernouilli random variables governing service success and joining decisions are shared between the two queues. The two queues then evolve together until the customers complete their first service attempt, either successfully departing or rejoining the queue. The possible sample path scenarios up to the first service completion are as follows:

\begin{itemize}
\item Case (i) The $Q_2$ process is killed before the tagged customer reaches position $i_1$; her number of visits to position 1 is zero.
\item Case (ii) The customer in $Q_2$ reaches position $i_1$, and both processes are killed before the tagged customer reaches position 1.
\item Case (iii) Both customers reach position 1, complete their first service successfully and leave the system.
\item Case (iv(a)) Both customers experience a failed service, rejoin their respective queues, and the queue lengths of $Q_1$ and $Q_2$ are equal.
\item Case (iv(b)) Both customers rejoin the queues after a failed service, but the queue length of $Q_1$ is strictly less than that of $Q_2$.
\end{itemize}

In Cases (i)–(iv(a)), it is clear that the number of visits to position 1 by the customer in $Q_2$ is less than or equal to that in $Q_1$. Case (iv(b)) mirrors the initial setting, and we can apply the same argument recursively until the customer eventually leaves the system.
\end{proof}

Noting the fact mentioned after equation (\ref{eq:d}) that the \( (i, j) \)th component of the vector \( \big(I - P_{\alpha}^{(x)} \big)^{-1} \bm{d}\) gives the expected total number of visits to states of the form \((1,\ell)\) before the process is absorbed, conditional on the initial state being \( (i, j) \), we immediately have a relationship between the expected discounted rewards. 

\begin{corollary} \label{corollary:1}
    
\[
\left( \left(I - P_{\alpha}^{(x)}\right)^{-1} \bm{d} \right)_{i_1, i_1} \geq \left( \left(I - P_{\alpha}^{(x)}\right)^{-1} \bm{d} \right)_{i_2, i_2} \,,
\]
which implies that
\[
\mathbb{E} \left( e^{-\alpha W_{i_1,i_1}^{(x)}} \right) \geq \mathbb{E} \left( e^{-\alpha W_{i_2,i_2}^{(x)}} \right).
\]
\end{corollary}

\begin{lemma}\label{lemma:2}
For integer $n \geq i$, the random variable \( N^{(x)}_{i,n} \) stochastically dominates \( N^{(x)}_{i,n+1} \).
\end{lemma}

\begin{proof}
Assume there are two queues, $Q_1$ and $Q_2$, where a customer is at the same position $i$ in both queues. Queue $Q_1$ contains $n$ customers, while $Q_2$ contains $n+1$ customers. In the coupling setup, the arrival times, service durations, the killing random variable and the Bernouilli random variables governing service success and joining decisions are identical for both queues. The sample paths up to the first service completion fall into four possible cases:

\begin{itemize}
\item Case (i) The process is killed before the tagged customer reaches position 1.
\item Case (ii) The customer reaches position 1, completes service successfully, and exits the system.
\item Case (iii) The customer reaches position 1, but the service fails. She rejoins the queue, and the lengths of $Q_1$ and $Q_2$ are now equal. This situation arises when the number of customers in \( Q_2 \) reaches \( \lfloor x \rfloor + 1 \). Then if another customer arrives, she will join \( Q_1 \) but not \( Q_2 \). This  balances the lengths of the two queues before the tagged customer's service fails and she has to rejoin.
\item Case (iv) The customer reaches position 1, the service fails, and upon rejoining, the queue length of $Q_1$ is one less than that of $Q_2$.
\end{itemize}

In Cases (i)-(iii), the number of visits the tagged customer makes to position 1 in $Q_1$ is equal to that in $Q_2$. Case (iv) corresponds to Case (ivb) in the proof of Lemma~\ref{lemma:1}.

\end{proof}
From Lemma~\ref{lemma:2}, we have
\begin{corollary} \label{corollary:2}
\[
\left( \left(I - P_{\alpha}^{(x)}\right)^{-1} \bm{d} \right)_{i, n} \geq \left( \left(I - P_{\alpha}^{(x)}\right)^{-1} \bm{d} \right)_{i, n+1} \,,
\]
which implies that
\[
\mathbb{E} \left( e^{-\alpha W_{i, n}^{(x)}} \right) \geq \mathbb{E} \left( e^{-\alpha W_{i, n+1}^{(x)}} \right).
\]
\end{corollary}

Using similar reasoning as in the proof of \cite[Lemma 2]{FTW21}, it follows from Corollaries \ref{corollary:1} and \ref {corollary:2} that
\[
\mathbb{E}\left(e^{-\alpha W_{i,j}^{(x_1)}} \right) \geq \mathbb{E}\left(e^{-\alpha W_{i,j}^{(x_2)}} \right) \quad \text{for } x_1 \leq x_2.
\]
We now return to interpreting \(\alpha\) as a discount factor. To simplify our notation, we define
$$
z_{i,j}^{(x)}(\alpha) :=  \mathbb{E}\left( \,e^{-\alpha W_{i,j}^{(x)}(\alpha)}\right) -v.
$$
The following lemma will be useful for establishing the uniqueness of the Nash equilibrium.
\begin{lemma}
If
$$
z_{m+1,m+1}^{(m+1)} (\alpha) < 0 < z_{m+1,m+1}^{(m)}(\alpha),
$$
then there exists a unique $x \in (m, m+1)$ such that
$$
z_{m+1,m+1}^{(x)}(\alpha) = 0.
$$
\end{lemma}
\begin{proof}
Since $I - P_{\alpha}^{(x)}$ is non-singular, the argument of \citet[Lemma 3]{FTW21} can be adapted to show that $z_{m+1,m+1}^{(x)}(\alpha)$ is continuous and strictly decreasing in $x \in (m, m+1)$, completing the proof.
\end{proof}

\begin{theorem} \label{theorem: NaEx}
The Nash equilibrium threshold $x_e$ is given by
$$
x_e =
\begin{dcases}
0, & \text{if } z_{1,1}^{(0)} (\alpha)< 0, \\[6pt]
\zeta_0, & \text{if } z_{1,1}^{(0)}(\alpha) = 0, \\[6pt]
m, & \text{if } z_{m+1,m+1}^{(m)}(\alpha) \leq 0 \leq z_{m,m}^{(m)}(\alpha), \quad m = 1, 2, \ldots, \\[6pt]
\chi_m, & \text{if } z_{m+1,m+1}^{(m+1)}(\alpha) < 0 < z_{m+1,m+1}^{(m)}(\alpha), \quad m = 1, 2, \ldots,
\end{dcases}
$$
where $\zeta_0$ is any number in $[0,1]$, and $\chi_m := \{ x : z_{m+1,m+1}^{(x)}(\alpha) = 0 \}$.
\end{theorem}
\begin{proof}
    \begin{itemize}
    \item  If $z_{1,1}^{(0)} (\alpha)< 0$, even when no customers are in the system, it is not optimal for the tagged customer to join the system. Hence, not joining is the best response for any customer given that this is used by the other customers . 
    \item If \( z_{1,1}^{(0)} (\alpha)= 0 \), then \( z_{1,1}^{(\zeta_0)}(\alpha) = 0 \) for all \( \zeta_0 \in [0,1] \). That is, given that others use a threshold \( \zeta_0 \), the tagged customer is indifferent between joining and balking. Thus, any threshold in \( [0,1] \) is a best response, and hence any value in \( [0,1] \) constitutes a Nash equilibrium threshold.
    \item If \( z_{m+1,m+1}^{(m)}(\alpha) \leq 0 \leq z_{m,m}^{(m)}(\alpha) \), when all other customers use a threshold \( m \), it is optimal for the tagged customer to do the same. Therefore, threshold \( m \) is a Nash equilibrium.
    \item If $z_{m+1,m+1}^{(m+1)}(\alpha) < 0 < z_{m+1,m+1}^{(m)}(\alpha)$, then the best response of the tagged customer is to choose a threshold greater than $m$ when the other customers use threshold $m$, and to choose a threshold less than $m+1$ when the other customers use threshold $m+1$. Moreover, when all others adopt threshold $\chi_m$, the tagged customer is indifferent between joining at position $m$ or $m+1$. Therefore, any threshold in the interval $[m, m+1]$, including $\chi_m$, is a best response. Hence, $\chi_m$ is a Nash equilibrium threshold.
    \end{itemize}
\end{proof}

\section{The Case Where Customers Can Renege} \label{sec:ExRcase}

In this section, we examine customers' strategic behavior when they are allowed to renege after joining the system. Specifically, arriving customers observe the current number of customers and decide whether to join based on a threshold strategy. If their service attempt fails, they return to the end of the queue and make a new joining decision using the same threshold rule. When the other customers follow threshold $x$, there can be no more than $\lfloor x \rfloor + 1$ other customers in the queue when the tagged customer arrives, and so it suffices to compute the expected payoff of a tagged customer who joins in positions $1 \leq i \leq \lfloor x \rfloor + 1$. In fact, it turns out that the value of $i$ that is relevant for calculating the Nash equilibrium is $i = \lfloor x \rfloor + 1$. For further details, see \citet[Section 4]{FTW21}.

In the $R$-case, let $\widehat{W}_{i,j}^{(\lfloor x \rfloor + 1, x)}$ denote the remaining time until the tagged customer either successfully completes service or reneges, given that there are $j$ customers in the system, the tagged customer is in position $i$ using threshold $\lfloor x \rfloor + 1$, and all other customers follow a threshold strategy with threshold $x$. The corresponding expected payoff is defined by
\[
\widehat{z}_{i,i}^{(\lfloor x \rfloor +1, x)}(\alpha) := \mathbb{E}\left( \, e^{-\alpha \widehat{W}_{i,i}^{(\lfloor x \rfloor +1, x)}}\right) - v.
\]
For $t \geq 0$, let $\widehat{X}(t)$ and $\widehat{T}(t)$ denote the analogues of $X(t)$ and $T(t)$ in the $N$-case. Then, for $1 \leq i \leq j \leq \lfloor x \rfloor + 1$, the conditional probability
{\footnotesize
\begin{align}
&P\{\widehat{X}(\widehat{T}(t)) = (i',j')  \mid \widehat{X}(t) = (i,j)\}  \\
&\qquad= \begin{dcases}
\frac{\lambda}{\lambda+\mu} & \text{if} \quad j < \lfloor x \rfloor, \, (i',j') = (i,j+1) \text{ or } j > \lfloor x \rfloor, \, (i',j') = (i,j) \\
\frac{\lambda p}{\lambda+\mu} & \text{if} \quad j = \lfloor x \rfloor, \, (i',j') = (i,j+1) \\
\frac{\lambda (1-p)}{\lambda+\mu} & \text{if} \quad  j = \lfloor x \rfloor, \, (i',j') = (i,j) \\
\frac{\mu q}{\lambda+\mu} & \text{if} \quad  i >1, \, j < \lfloor x \rfloor +1, \,  (i',j') = (i-1,j-1) \text{ or } i =1, \, (i',j') = 0 \\
\frac{\mu (1-q)}{\lambda+\mu} & \text{if} \quad  i >1, \, j < \lfloor x \rfloor +1, \, (i',j') = (i-1,j) \text{ or }  i =1, \, (i',j') = (j,j) \\
\frac{\mu (q + (1-q)(1-p))}{\lambda+\mu} & \text{if} \quad  i >1, \, j = \lfloor x \rfloor +1, \,  (i',j') = (i-1,j-1) \\
\frac{\mu (1-q)p}{\lambda+\mu} & \text{if} \quad  i >1, \, j = \lfloor x \rfloor +1, \, (i',j') = (i-1,j) \,.
\end{dcases} \notag
\end{align}}
Similar to Equation \eqref{eq:Wij}, we assume that the system is in state $(i,j)$ at time $t$. Conditioning on the first transition out of state $(i,j)$, the expected discount factor $\mathbb{E}\left(e^{-\alpha \widehat{W}_{i,i}^{(\lfloor x \rfloor +1, x)}}\right)$ can be expressed as in Appendix \ref{appendix:wijR}.

Hence, the vector of expected discount factors
\begin{equation} \label{eq:totalDisFactor2}
\bm{\widehat{\gamma}}^{(\lfloor x \rfloor+1, x)}_{\alpha} \equiv \left[ \mathbb{E} \left(e^{-\alpha\widehat{W}^{(\lfloor x \rfloor+1, x)}_{1,1}} \right), \mathbb{E} \left(e^{-\alpha\widehat{W}^{(\lfloor x \rfloor+1, x)}_{1,2}} \right),\mathbb{E} \left(e^{-\alpha\widehat{W}^{(\lfloor x \rfloor+1, x)}_{2,2}} \right), \ldots, \mathbb{E} \left(e^{-\alpha\widehat{W}^{(\lfloor x \rfloor+1, x)}_{\lfloor x \rfloor+1,\lfloor x \rfloor+1}} \right) \right]^T
\end{equation}
can be obtained by solving Poisson's equation
\begin{equation} \label{eq:EPoisson2}
\left(I - \widehat{P}^{(\lfloor x \rfloor +1,x)}_\alpha\right) \bm{\widehat{\gamma}^{(\lfloor x \rfloor+1, x)}_{\alpha}} = \bm{g}_\alpha
\end{equation}
where $\widehat{P}^{(\lfloor x \rfloor +1,x)}_\alpha$ and $\bm{g}_\alpha$ are defined in Appendix \ref{appendix: 2.1} and Equation \eqref{eq:g}, respectively.
The algorithm for solving Equation \eqref{eq:EPoisson2} is nearly identical to Algorithm \ref{alg:Poisson}, with the only difference being that the matrices $U^{j}, \Gamma^{j}, G^{(j)}$ are now derived from $\widehat{P}^{(\lfloor x \rfloor +1,x)}_\alpha$.

Next, we examine how allowing reneging affects customer decisions. The first step is to compare ${z}_{j,j}^{(x)}(\alpha)$ and $\widehat{z}_{j,j}^{(\lfloor x \rfloor +1, x)}(\alpha)$.
\begin{lemma} \label{lemma:ECompareZ}
For all thresholds $x$ and for all joining positions $j$, the expected payoff in the $R$-case is at least as high as in the $N$-case, that is
\begin{equation}
\widehat{z}_{j,j}^{(\lfloor x \rfloor +1, x)}(\alpha) \geq z_{j,j}^{(x)}(\alpha) \qquad j = 1, 2, \ldots, \lfloor x \rfloor+1 \,.
\end{equation}
When $x$ is an integer, 	
\begin{equation}
\widehat{z}_{j,j}^{(x, x)}(\alpha)  = \widehat{z}_{j,j}^{(x +1, x)}(\alpha) = z_{j,j}^{(x)}(\alpha) \qquad j = 1, 2, \ldots, \lfloor x \rfloor \,.
\end{equation}
\end{lemma}
\begin{proof}
Equations \eqref{eq:EPoisson} and \eqref{eq:EPoisson2} give us
\begin{align*}
\left(I-P_{\alpha}^{(x)}\right) \, \bm{\gamma}^{(x)}_\alpha & = \bm{g}_\alpha \\
\left(I-\widehat{P}_{\alpha}^{(\lfloor x \rfloor +1, x)}\right) \, \bm{\widehat{\gamma}}^{(\lfloor x \rfloor+1, x)}_{\alpha}& = \bm{g}_\alpha \,.
\end{align*}
Rearranging terms, we obtain:
\begin{align}
\left(I - \widehat{P}_{\alpha}^{(\lfloor x \rfloor +1, x)}\right) \, \left( \bm{\widehat{\gamma}}^{(\lfloor x \rfloor+1, x)}_{\alpha}- \bm{\gamma}^{(x)}_\alpha\right) & = \left(\widehat{P}_{\alpha}^{(\lfloor x \rfloor +1, x)} - P_\alpha^{(x)} \right) \, \bm{\gamma}^{(x)}_\alpha \\
& = \frac{\mu (1-q) (1-p)}{\lambda+\mu+\alpha}
\begin{bmatrix}
    \bm{0}_{\frac{n(n+1)}{2}+ 1} \\
	\mathbb{E} \left(e^{-\alpha W_{1,n}}\right) - \mathbb{E} \left(e^{-\alpha W_{1,n+1}}\right) \\
	\mathbb{E} \left(e^{-\alpha W_{2,n}}\right) - \mathbb{E} \left(e^{-\alpha W_{2,n+1}}\right) \\
	\vdots \\
	\mathbb{E} \left(e^{-\alpha W_{n,n}}\right) - \mathbb{E} \left(e^{-\alpha W_{n,n+1}}\right) \\
\end{bmatrix}
\end{align}
and so 
\begin{equation}
\label{eq:gammadiff}
\left( \bm{\widehat{\gamma}}^{(\lfloor x \rfloor+1, x)}_{\alpha}- \bm{\gamma}^{(x)}_\alpha\right) =  \frac{\mu (1-q) (1-p)}{\lambda+\mu+\alpha} \left(I - \widehat{P}_{\alpha}^{(\lfloor x \rfloor +1, x)}\right)^{-1}
\begin{bmatrix}
    \bm{0}_{\frac{n(n+1)}{2}+ 1} \\
	\mathbb{E} \left(e^{-\alpha W_{1,n}}\right) - \mathbb{E} \left(e^{-\alpha W_{1,n+1}}\right) \\
	\mathbb{E} \left(e^{-\alpha W_{2,n}}\right) - \mathbb{E} \left(e^{-\alpha W_{2,n+1}}\right) \\
	\vdots \\
	\mathbb{E} \left(e^{-\alpha W_{n,n}}\right) - \mathbb{E} \left(e^{-\alpha W_{n,n+1}}\right) \\
\end{bmatrix}.
\end{equation}
The $((i_1,j_1),(i_2,j_2))$-th entry of $\left(I - \widehat{P}_{\alpha}^{(\lfloor x \rfloor +1, x)}\right)^{-1}$ represents the expected number of visits to state $(i_2,j_2)$ starting from state $(i_1,i_2)$ before the tagged customer exits the system or the killing random variable stops the process. We can conclude from this that $\left(I - \widehat{P}_{\alpha}^{(\lfloor x \rfloor +1, x)}\right)^{-1} >0$. Furthermore, it follows from Corollary \ref{corollary:2} that the column on the right hand side of (\ref{eq:gammadiff}) is nonnegative. It follows that
\[
\left( \bm{\widehat{\gamma}}^{(\lfloor x \rfloor+1, x)}_{\alpha}-  \bm{\gamma}^{(x)}(\alpha)\right) \geq 0 \,.
\]
Consequently, for all $j = 1, 2, \ldots, \lfloor x \rfloor+1$, we have $\widehat{z}_{j,j}^{(\lfloor x \rfloor +1, x)}(\alpha) \geq z_{j,j}^{(x)}(\alpha)$.

When $x$ is an integer, the proof follows similarly to that of \citet[Lemma 4]{FTW21}, and we omit the details here.
\end{proof}

\begin{theorem} \label{theorem: ExNaCom}
	The Nash equilibrium $\widehat{x}_e$ in the $R$-case is 
	\begin{equation}
	\widehat{x}_e = 
	\begin{dcases}
	0 & \text{if } \, \widehat{z}_{1,1}^{(1, 0)} (\alpha) < 0 \\
	\zeta_0 & \text{if } \, \widehat{z}_{1,1}^{(1, 0)} (\alpha) = 0 \\
	m & \text{if } \, \widehat{z}_{m+1,m+1}^{(m+1, m)} (\alpha) \leq 0 \leq   \widehat{z}_{m,m}^{(m, m)}(\alpha) \\
	\widehat{\chi}_m & \text{if } \, \widehat{z}_{m+1,m+1}^{(m+1, m+1) } (\alpha)< 0 < \widehat{z}_{m+1,m+1}^{(m+1, m)} (\alpha)\,,
	\end{dcases}
	\end{equation}
        where $\widehat{\chi}_m := \{ x : \widehat{z}_{m+1,m+1}^{(m+1, x)} (\alpha) = 0 \}$.
\end{theorem}
\begin{proof}
    The argument follows an approach similar to the proof of Theorem \ref{theorem: NaEx}, and we omit the details here.
\end{proof}


\begin{figure}[h]
	\centering{\includegraphics[width=0.5\linewidth]{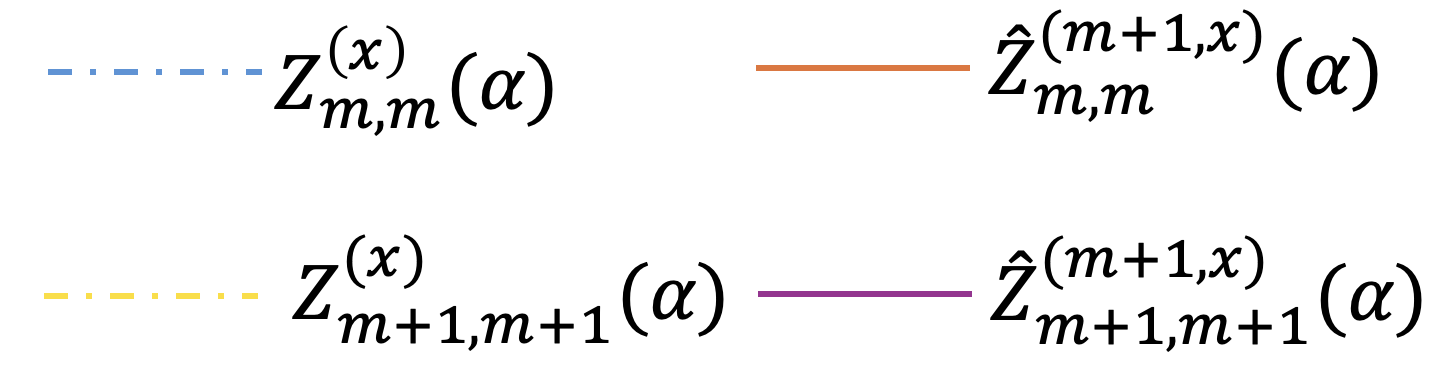}}\\
	\subcaptionbox{}
	{\includegraphics[width=0.323\linewidth]{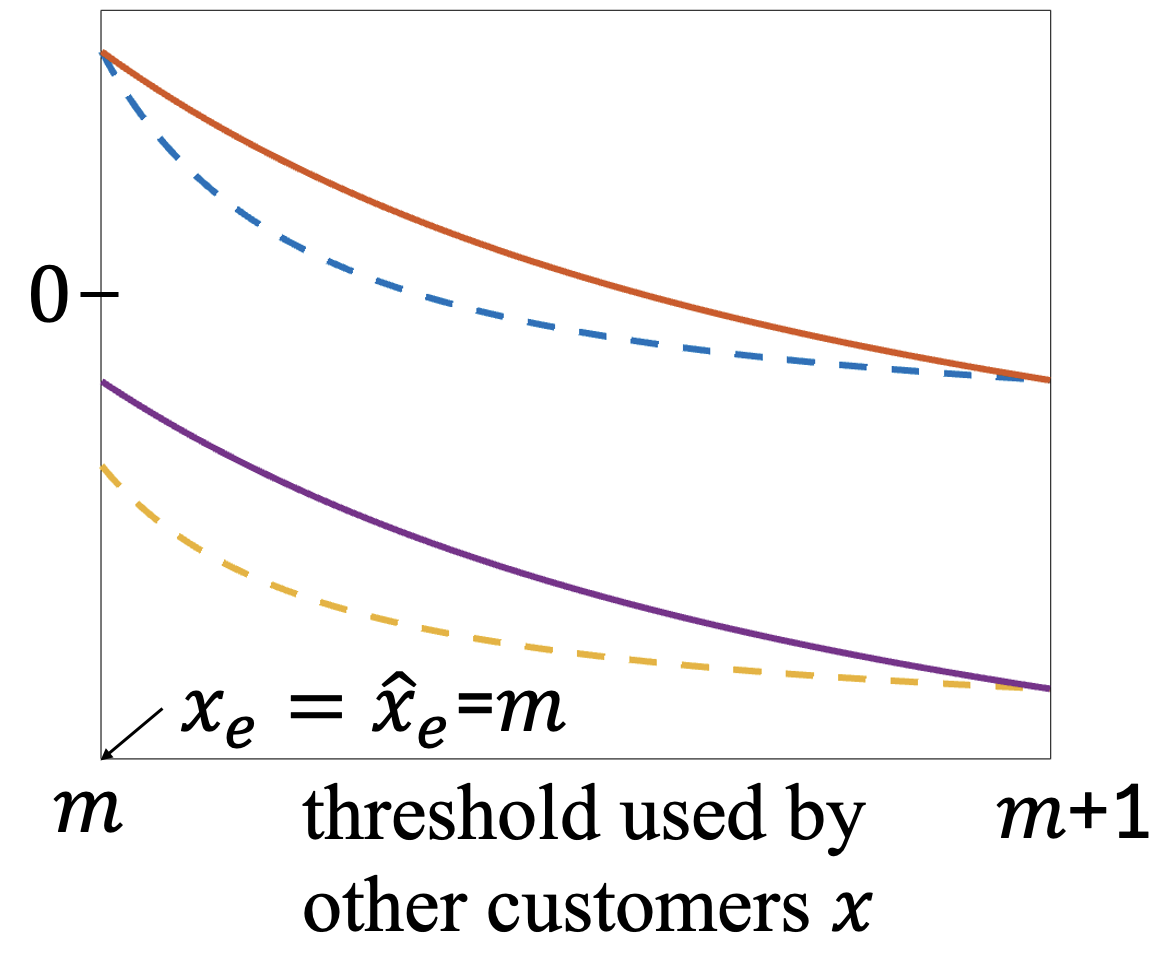}}  
	\hspace{\fill}
	\subcaptionbox{}
	{\includegraphics[width=0.323\linewidth]{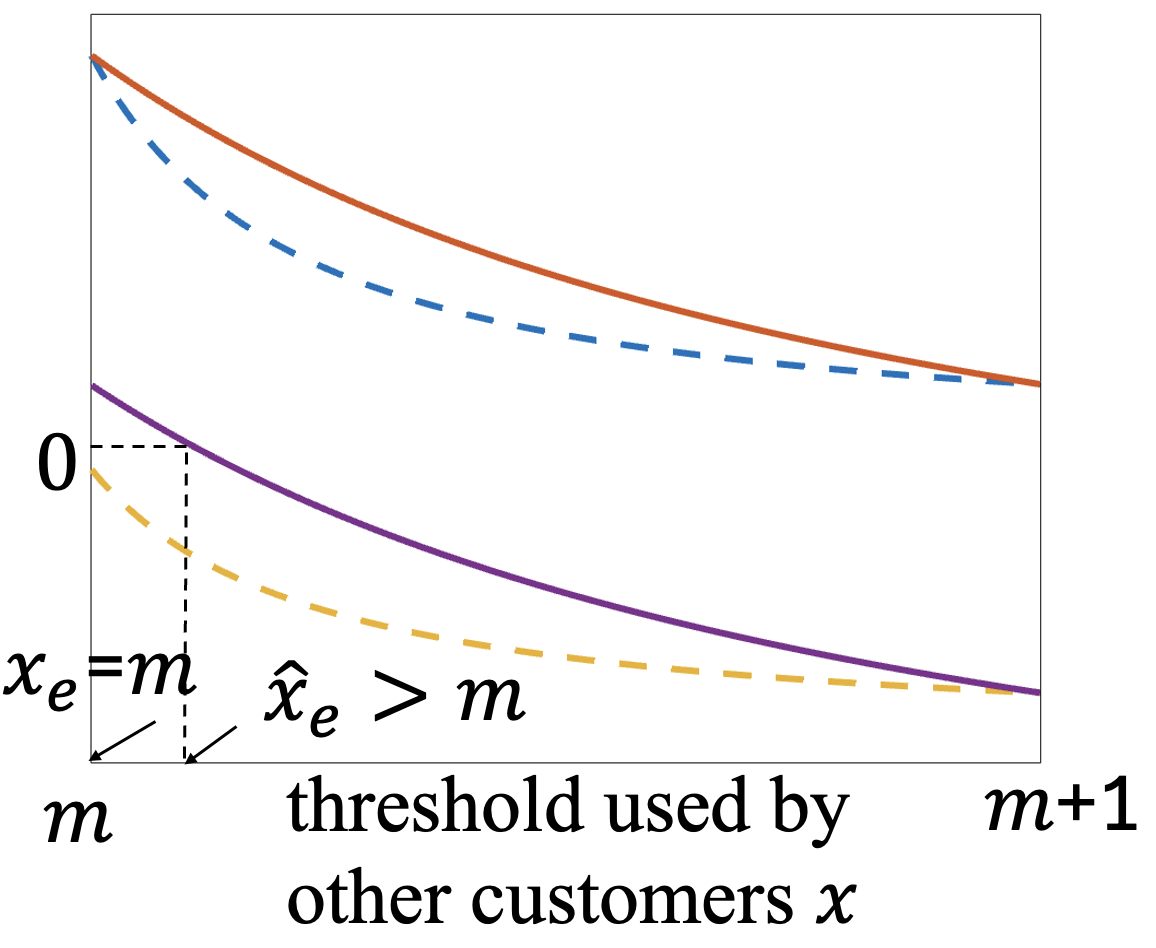}} 
	\hspace{\fill}
	\subcaptionbox{}
	{\includegraphics[width=0.323\linewidth]{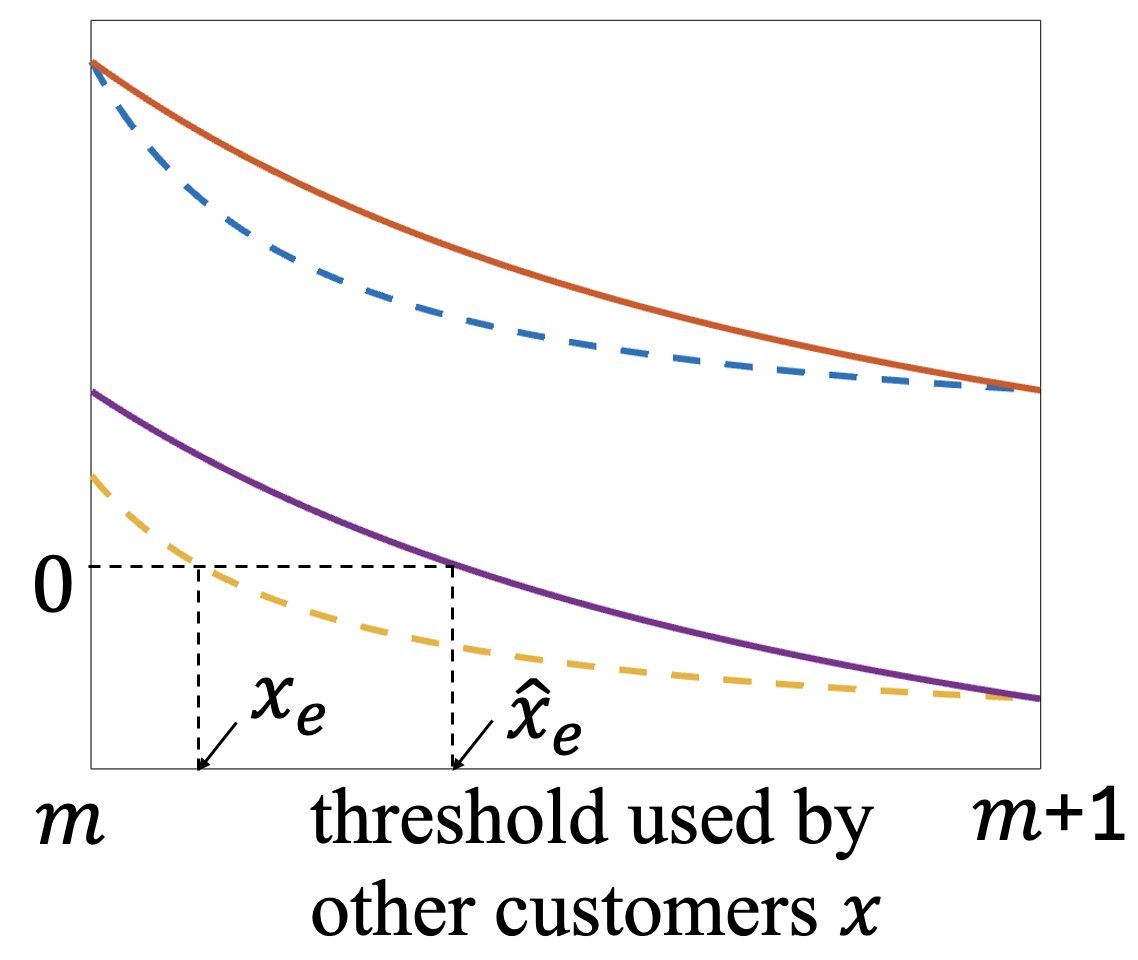}} 
	\caption{An illustration of Nash equilibrium threshold comparison.} \label{fig:NEE}
\end{figure}  
The following corollary compares $\widehat{x}_e$ and $x_e$, with Figure \ref{fig:NEE} illustrating the differences across three distinct scenarios.
\begin{corollary}
    Under identical parameter settings, the equilibrium threshold satisfies $\widehat{x}_e \geq x_e$, that is, customers have a stronger joining incentive when reneging is permitted.
\end{corollary}
\begin{proof}
    \begin{itemize}
		\item When $ \widehat{z}_{m+1,m+1}^{(m+1,m)} (\alpha)  \leq 0 \leq   z_{m,m}^{(m)}(\alpha)$,
		\begin{equation}
		\widehat{z}_{m,m}^{(m,m)} (\alpha)=  {z}_{m,m}^{(m)}(\alpha) \geq 0  \,,
		\end{equation}
		and
		\begin{equation}
		{z}_{m+1,m+1}^{(m)} (\alpha)< \widehat{z}_{m+1,m+1}^{(m+1,m)} (\alpha) \leq 0 \,.
		\end{equation}
		Hence in both the $N$-case and the $R$-case, $x_e=\widehat{x}_e = m$. This case is represented in Figure \ref{fig:NEE}(a).
		
		\item When $ z_{m+1,m+1}^{(m)} (\alpha)< 0 \leq \widehat{z}_{m+1,m+1}^{(m+1,m)}(\alpha)$,
		\begin{equation}
		{z}_{m,m}^{(m)} (\alpha)	= \widehat{z}_{m,m}^{(m,m)} (\alpha)> \widehat{z}_{m+1,m+1}^{(m+1,m)}(\alpha)  > 0 \,,
		\end{equation}
		and 
		\begin{equation}
		\widehat{z}_{m+1,m+1}^{(m+1,m+1)}(\alpha) < z_{m+1,m+1}^{(m)}(\alpha) \leq 0 \,.
		\end{equation}
		Hence, $x_e= m < \widehat{x}_e < m+1$. This case is represented in Figure \ref{fig:NEE}(b).
		
		\item When $\widehat{z}_{m+1,m+1}^{(m+1,m+1)}(\alpha) < 0 <   z_{m+1,m+1}^{(m)}(\alpha)$,
		\begin{equation}
		{z}_{m+1,m+1}^{(m+1)}(\alpha) = \widehat{z}_{m+1,m+1}^{(m+1,m+1)}(\alpha) < 0
		\end{equation}
		and
		\begin{equation}
		\widehat{z}_{m+1,m+1}^{(m+1,m)}(\alpha)  > {z}_{m+1,m+1}^{(m)}(\alpha) > 0 \,.
		\end{equation}
		Hence, $m <x_e< \widehat{x}_e < m+1$. This case is represented in Figure \ref{fig:NEE}(c).
	\end{itemize}
\end{proof}

\section{The Effects of Decreasing the Discount Rate and the Value of the Outside option} \label{sec:ExParadox}

The parameter \( \alpha \) represents the discount rate. A customer who has experienced a waiting time $w$ before receiving a successful service will receive a higher reward if \( \alpha \) is smaller. Similarly, she will not have to pay as much to enter the system if the value $v$ of the outside option is smaller. Intuitively, we would expect both of these things to be good for customers. 

However, in this section, we demonstrate that there are parameter settings where decreasing \(\alpha \) and/or \( v \) decreases the stationary expected individual payoff, as defined in Equation~\eqref{eq:payoff}. 
We would also expect that permitting reneging ought to be a good thing for the customers. However, this too can result in a decrease of the stationary expected individual payoff. Such phenomena can be thought of as paradoxical in the sense that Braess's paradox \cite{brae68} is paradoxical. 

The first paradox arises when we examine the effect on the stationary individual payoff of decreasing \( \alpha \) or \( v \).

\begin{paradox} \label{paradox1a}
Despite our intuition that these measures should be good for customers, a decrease in \( \alpha \) or \( v \) can paradoxically lead to a decrease in the stationary individual payoff \( V(x_e) \).
\end{paradox}

\begin{table}[h!]
\begin{center}
{\footnotesize
\begin{tabular}{|l|c|c|c|c|} 
\hline
\multirow{2}{*}{$v = 0.5$} &\multicolumn{4}{|c|}{$\alpha$} \\
\cline{2-5}
& $0.1$ & $0.075$ & $0.05$ & $0.025$  \\
\hline
$x_e$ & 1 & 1.48 & 2.21 & 5 \\ 
\cline{1-5}
$V(x_e)$ & 0.0055 & 0.0033 & 0.0048 & 0.0046 \\
\hline
\multirow{2}{*}{$\alpha = 0.05$} &\multicolumn{4}{|c|}{$v$} \\
\cline{2-5}
& $0.62$ &  $0.6$ & $0.5$ & $0.49$  \\
\hline
$x_e$ &  1.13 & 1.85 & 2.22 & 2.53  \\ 
\cline{1-5}
 $V(x_e)$ & 0.0073 & 0.0016 & 0.0047 & 0.0024 \\
\hline
\end{tabular}}
\caption{The Nash equilibrium thresholds and the stationary individual payoffs when $\lambda = 1, \, \mu = 0.5$.} \label{tab:table0}
\end{center}
\end{table}

Table \ref{tab:table0} presents the Nash equilibrium thresholds alongside the corresponding stationary individual payoffs. As $\alpha$ or $v$ decreases, the threshold $x_e$ increases. However, the stationary individual payoff may, counterintuitively, decrease. For the case where the payoff is a linear function of the waiting time,a similar phenomenon was observed in \citet[Paradox 1]{FTW21}. This occurs because a higher equilibrium threshold leads to more customers joining the system, which increases congestion and hence the expected sojourn time for all individuals ultimately making everyone worse off. 

Our next paradox concerns the effect of allowing reneging on the stationary individual payoff. While reneging introduces more flexibility for customers and leads to an increase in the equilibrium threshold, it is not clear whether this results in a higher or lower stationary individual payoff.

\begin{paradox}
 When the Nash equilibrium threshold in the $N$ case is an integer, the stationary expected individual payoff for the $R$-case is the same as in the $N$-case. When the Nash equilibrium is not an integer, the stationary expected individual payoff for the $R$-case can be lower than in the $N$-case.
\end{paradox}

\begin{theorem}
    When $x_e = \widehat{x}_e = m$, we have 
    \[
    V(x_e) = V(\widehat{x}_e) \,;
    \]
    whereas when $x_e= m < \widehat{x}_e < m+1$, then it is possible that
    \[
    V(x_e) > V(\widehat{x}_e) \,.
    \]
\end{theorem}
\begin{proof}
    The proof of the first part is similar to that for \citet[Paradox 2]{FTW21}, so we omit the details here. 

    For two sets of parameter values, Table \ref{tab:paradox2} illustrates the equilibrium thresholds, the equilibrium payoffs conditional on the tagged customer's joining state and the stationary expected individual payoffs. We observe that the stationary expected individual payoff is smaller when reneging is allowed than when it is not. 

    \begin{table}[H]
	\caption{Numerical examples with parameters $R = 2$, $q = 0.2$, and $v = 1$.} 
	\label{tab:paradox2}
	\centering
	\begin{tabular}{|l | c  | c | } 
		\hline
		\quad $\alpha, \, \lambda, \, \mu$ & $0.05, 0.4, 0.7 $ & $0.04, 0.4, 0.55$ \, \\ [0.5ex] 
		\hline 
		$x_e\qquad\qquad \, \widehat{x}_e$ & $2.37 \quad 2.84$ & $2.17\quad2.70$ \\
		\hline
		$z_{1,1}^{(x_e)}(\alpha)\qquad \widehat{z}_{1,1}^{(\widehat{x}_e, \widehat{x}_e)} (\alpha)$ & $ 0.29 \quad 0.27$ & $0.28 \quad0.25$ \\
		\hline
		$z_{2,2}^{(x_e)}(\alpha)\qquad \widehat{z}_{2,2}^{(\widehat{x}_e, \widehat{x}_e)} (\alpha)$ & $0.12 \quad 0.09$  & $0.13 \quad 0.09$ \\
		\hline
		$z_{3,3}^{(x_e)}(\alpha)\qquad \widehat{z}_{3,3}^{(\widehat{x}_e, \widehat{x}_e)} (\alpha)$ & $0 \quad 0$  & $0 \quad 0$ \\
		\hline
        stationary expected individual payoff & $0.034 \quad 0.022$    & $0.028 \quad 0.017$ \\
		\hline
	\end{tabular}
\end{table}	
\end{proof}

In our numerical experiments for cases where $m <x_e< \widehat{x}_e < m+1$, we were unable to find a set of parameter values where the stationary expected individual payoff was not lower for the $R$-case than for the $N$-case. This leads us to conjecture that $V(x_e)$ is always greater than $V(\widehat{x}_e)$ when $m <x_e< \widehat{x}_e < m+1$. We are, however, unable to prove this in general. 


\section{The sojourn time distribution} \label{sec:ExTimeDist}

In this section, we follow the third interpretation described in \ref{threeroles} and treat $\alpha$ as a Laplace transform variable. We first derive the Laplace transform of the sojourn time and then numerically invert it to obtain the corresponding distribution. To illustrate how this distribution can be used, we present an example involving a payoff that depends on the expectation of a non-standard functional of the waiting time. Since the procedure for the $N$-case is similar to that for the $R$-case, we focus exclusively on the $N$-case.

We begin by defining
\[
F^{(x)}_{i,j}(w) \equiv \mathbb{P}(W^{(x)}_{i,j} \leq w), \qquad 1 \leq i \leq j \leq \lfloor x \rfloor + 1,
\]
as the distribution function of $W^{(x)}_{i,j}$ in the $N$-case. For $\Re(\alpha) \geq 0$, the Laplace transform of this distribution is given by
\begin{align}
\widetilde{F}^{(x)}_{i,j}(\alpha) \equiv \mathcal{L}(F^{(x)}_{i,j}(t)) := \int_{0}^{\infty} e^{-\alpha t} \, dF^{(x)}_{i,j}(t) = \mathbb{E} \left( e^{-\alpha W_{i,j}^{(x)}} \right), \label{eq: Laplace1}
\end{align}
which matches Equation~\eqref{eq:Wij}.

The distributions of the $W^{(x)}_{i,j}$ can be obtained by numerically inverting the Laplace transforms $\widetilde{F}^{(x)}_{i,j}(\alpha)$, using techniques such as those described in \citet{K21}. Once these distributions are available for $1 \leq i \leq j \leq \lfloor x \rfloor$, any expected payoff that depends on the sojourn time can be computed, assuming that arriving customers decide to join the queue according to a threshold policy. Both the payoff studied in \citet{FTW21} and the discounted reward considered in this work fall into this category.
In the following, we demonstrate how the sojourn time distribution can be applied through an example involving a different payoff setting.

\citet{HH95} considered a setting in which an individual's reward drops to zero if their waiting time exceeds a fixed threshold \( \Xi \). Specifically, a customer receives a reward \( R \) if her service is completed within \( \Xi \) time units; otherwise, the reward is zero. We shall consider a situation where customers arriving to the feedback queue are subject to a similar reward structure.  

Let
\[
\phi_i^{(x)}(\Xi) \equiv \mathbb{P}\left(W_i^{(x)} \leq \Xi \right),
\]
denote the probability that a customer's sojourn time is less than \( \Xi \), assuming all other customers adopt the threshold strategy \( x \), and she arrives in position \( i \). This quantity captures the customer's perceived likelihood of receiving service before the deadline \( \Xi \), beyond which her value drops to zero.
A customer chooses to join the system only if this probability exceeds her personal tolerance for delay-related risk, denoted by \( \gamma \in [0,1]\). This leads to a payoff criterion:
\begin{equation}\label{eq:payoff2}
   \phi_i^{(x)}(\Xi) - \gamma.
\end{equation}
According to the definition of a Nash equilibrium, if \( \phi_i^{(x)}(\Xi) > \gamma \), the customer prefers to join; if \( \phi_i^{(x)}(\Xi) = \gamma \), she is indifferent between joining and balking; and if \( \phi_i^{(x)}(\Xi) < \gamma \), she prefers to balk.

The Nash equilibrium threshold under this criterion is 
\begin{equation}
    x_e^{(\gamma,\Xi)} = 
    \begin{dcases}
        0 & \phi_{1}^{(0)}(\Xi) < \gamma \\
        \zeta_0 & \phi_{1}^{(0)}(\Xi) = \gamma \\
        m & \phi_{m+1}^{(m)}(\Xi) \leq \gamma \leq P_m^{(m)}(\Xi) \\
        \chi_m^{(\gamma,\Xi)} & \phi_{m+1}^{(m+1)}(\Xi) < \gamma < \phi_{m+1}^{(m)}(\Xi) 
    \end{dcases} \,,
\end{equation}
where $\chi_m^{(\gamma,\Xi)}:=\{x : \phi_{m+1}^{(x)}(\Xi)  = \gamma \}$.

 \begin{figure}[H]
	\centering
{\includegraphics[width=0.6\linewidth]{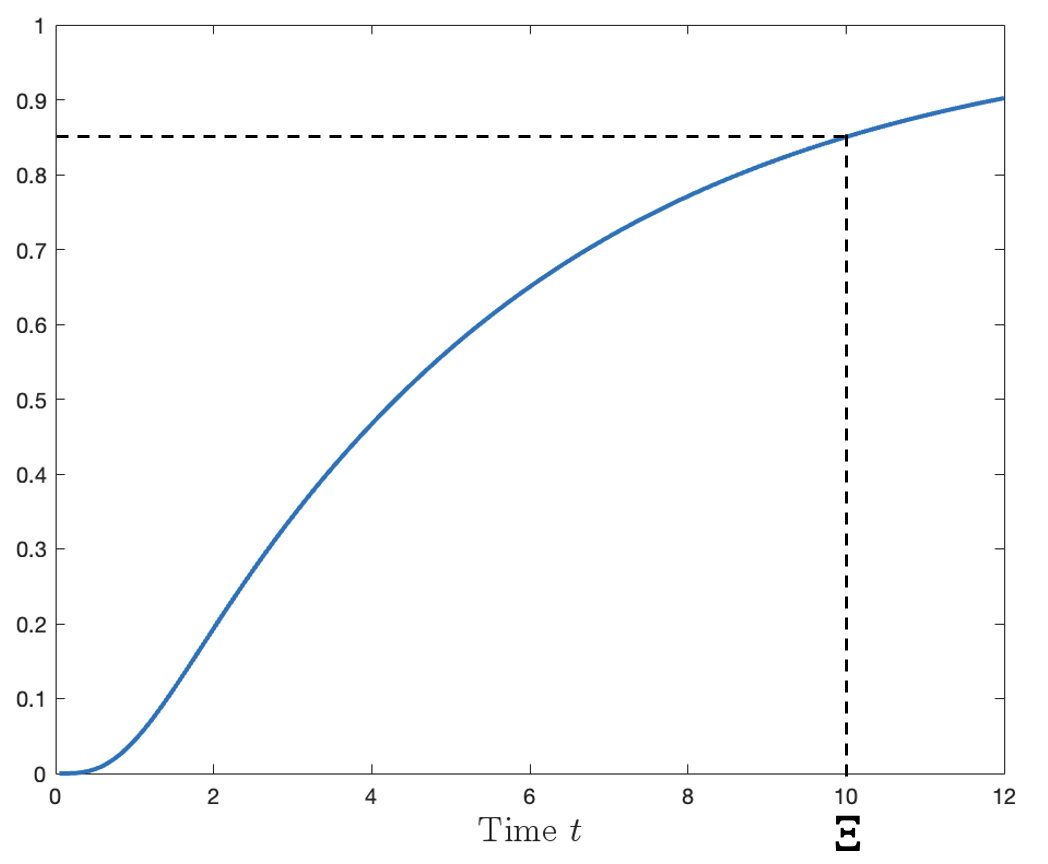}} 	\caption{The cumulative distribution function $\phi_{4}^{(3.6)}(\Xi)$ of the sojourn time for a tagged customer who joins in position $4$, while all other customers follow a threshold strategy with $x = 3.6$. The parameters are $\lambda = 1$, $\mu = 2$, and $q = 0.3$. In this example, the Nash equilibrium threshold is $x_e^{(0.85,10)} = 3.6$.
} \label{fig:SojournTimeD}
\end{figure}
We now use an example to illustrate how to identify a Nash equilibrium. 
Figure~\ref{fig:SojournTimeD} presents the cumulative distribution function $\phi_{4}^{(3.6)}(\Xi)$ of the sojourn time. Suppose $\gamma = 0.85$, then the figure shows that if a tagged customer joins in position 4, the probability that she will be served within $\Xi = 10$ units is 85\%.
Now, consider different thresholds used by others:
\begin{itemize}
    \item If the other customers use a threshold of 3, the system will become less crowded and $\phi_{4}^{(3)}(10) > 0.85$. Thus, the tagged customer will strictly prefer to join in position 4, and her best response is to use a threshold greater than $4$.
    \item If the other customers use a threshold of 4, the system will become more crowded and $\phi_{4}^{(4)}(10) < 0.85$. In this case, the best response of the tagged customer is to use a threshold $\leq 3$.
    \item If the other customers use a threshold of 3.6, the tagged customer's best response lies between 3 and 4, making her indifferent.
\end{itemize}
Therefore, 3.6 is the Nash equilibrium joining threshold.

The value of \( \Xi \) serves as a measure of how long the service quality persists before deteriorating to zero. A higher \( \Xi \) indicates that the service remains valuable for a longer duration. Intuitively, this should encourage more customers to join the system. This is confirmed by the increasing equilibrium threshold \( x_e \) as \( \Xi \) increases from 8 to 10 depicted in Table~\ref{tab:table1}.

However, Table~\ref{tab:table1} also reveals a paradox similar to that described in Paradox~\ref{paradox1a}. As shown in Table~\ref{tab:table1}, when \( \Xi \) increases from 9 to 10, the stationary individual payoff decreases, despite the longer-lasting service quality. This counterintuitive outcome arises because the improved service persistence attracts more customers, resulting in increased congestion and, ultimately, a lower stationary expected individual payoff.
\begin{table}[h!]
\begin{center}
{\footnotesize
\begin{tabular}{|l|l|c|c|c|c|} 
\hline
\multirow{2}{*}{$\gamma$}  & Nash equilibrium $\&$ &\multicolumn{3}{|c|}{$\Xi$} \\
\cline{3-5}
& stationary individual payoff & $8$ &  $9$ & $10$ \\
\hline
\multirow{2}{*}{$0.85$} 
&  $x_e$ &  3 & 3.03 & 3.61  \\ 
\cline{2-5}
&  $V(x_e)$ & 0.4742 & 0.8589 & 0.5957 \\
\hline
\end{tabular}}
\caption{The Nash equilibrium thresholds and the stationary individual payoffs for $\Xi = 8, \, 9, \, 10$, when $\lambda = 1, \, \mu = 2$.} \label{tab:table1}
\end{center}
\end{table}

Next, we focus on an example with \( \Xi = 10 \), and compare the Nash equilibrium thresholds under three different values of \( \gamma \): (1) \( \gamma = 0.8 \); (2) \( \gamma = 0.85 \); (3) \( \gamma = 0.9 \).
We then compute the Nash equilibrium threshold and the stationary individual payoff for \( q = 0.3 \) and \( q = 0.5 \). The key results are summarized in Table~\ref{tab:measure}. As shown, the equilibrium threshold \( x_e^{(\gamma, \Xi)} \) decreases monotonically as \( \gamma \) increases. This trend is intuitive: a higher service standard \( \gamma \) reduces the incentive to join.
The relationship between the stationary individual payoff and the service standard reveals an interesting threshold effect:
\begin{itemize}
    \item When the success probability is higher (\( q = 0.5 \)), customers achieve higher expected payoffs by adopting the stricter standard (\( \gamma = 0.9 \)).
    \item Conversely, when \( q = 0.3 \), the expected payoff is minimized at \( \gamma = 0.9 \), suggesting that employing an overly strict standard may be detrimental in lower-quality systems.
\end{itemize}
This non-monotonic behavior underscores the importance of accounting for the underlying probability distribution in strategic decision-making.

\begin{table}[http]
\centering
\caption{Nash equilibrium thresholds and corresponding stationary individual payoffs under different $q$ and $\gamma$, with $\Xi = 10$.} \label{tab:measure}
\begin{tabular}{|l|l|l|c|}
\hline
$q$ & $\gamma$ & Nash equilibrium & Stationary individual payoff \\
\hline
\multirow{3}{*}{0.3} & 0.8 & $x_e^{(0.8, 10)} = 4.05$ & 0.8042 \\
 & 0.85 & $x_e^{(0.85, 10)} = 3.61$  & 0.5957 \\
 & 0.9 & $x_e^{(0.9, 10)} =3$ & 0.5014 \\
 \hline
\multirow{3}{*}{0.5} & 0.8 & $x_e^{(0.8, 10)} = 8$ & 0.7983 \\
 & 0.85 & $x_e^{(0.85, 10)} =7$ & 0.8069 \\
 & 0.9 & $x_e^{(0.9, 10)} = 6$ & 0.8102\\
\hline
\end{tabular}
\end{table}

\section*{Acknowledgments}
\noindent P. G. Taylor's research is supported by the Australian Research Council (ARC) Laureate Fellowship FL130100039 and the ARC Centre of Excellence for the Mathematical and Statistical Frontiers (ACEMS). J. Wang would like to thank the University of Melbourne for supporting her work through the Melbourne Research Scholarship and the Albert Shimmins Fund. 


\begin{appendix}
	\section{The non-reneging case} \label{appendix:1.1}
	\begin{align}
	\begin{split}
	&P_{\alpha}^{(x)} = \begin{bmatrix}
	A_0^{(1)} & A_1^{(1)} & 0 & 0 & \cdots & \cdots & \cdots\\
	A_{-1}^{(2)} & A_0^{(2)} & A_1^{(2)}  & 0 & \cdots & \cdots & \cdots\\
	0 & A_{-1}^{(3)} & A_0^{(3)} & A_1^{(3)} & \cdots & \cdots & \cdots \\
	0 & 0 & A_{-1}^{(4)} & A_0^{(4)} & \cdots & \cdots & \cdots \\
	\vdots &\vdots & \vdots & \vdots & \ddots & \ddots &  \vdots \\
	\vdots &\vdots &\vdots & \vdots &  \vdots & A_{-1}^{(\lfloor x \rfloor + 1)} & A_0^{(\lfloor x \rfloor + 1)} \\
	\end{bmatrix} \,
	\end{split} \\
	\begin{split}
	&A_{-1}^{(k)} = \begin{bmatrix}
	0 & 0 & \cdots & 0\\
	\frac{\mu q}{\lambda+\mu+\alpha} & 0 & \cdots & 0 \\
	0 & \frac{\mu q}{\lambda+\mu+\alpha}  & \cdots & 0 \\
	\vdots & \vdots & \ddots & \vdots  \\
	0 & 0 & \cdots &  \frac{\mu q}{\lambda+\mu+\alpha}
	\end{bmatrix} \in \mathbb{R}^{k \times (k-1)} \, \qquad k = 2, \cdots, \lfloor x \rfloor + 1
	\end{split} \\
	\begin{split}
	&A_{0}^{(k)} = 
	\begin{bmatrix}
	0 & 0 & \cdots & \frac{\mu (1-q)}{\lambda+\mu+\alpha}\\
	\frac{\mu (1-q)}{\lambda+\mu+\alpha} & 0 & \cdots & 0 \\
	\vdots & \ddots & \ddots & \vdots  \\
	0 & \cdots & \frac{\mu (1-q)}{\lambda+\mu+\alpha} &  0
	\end{bmatrix}
	\in \mathbb{R}^{k \times k}
	\,  \quad k = 1, \cdots, \lfloor x \rfloor-1
	\end{split}\\
	\begin{split}
	&A_{0}^{(\lfloor x \rfloor)} = \begin{bmatrix}
	\frac{\lambda(1-(x-\lfloor x \rfloor))}{\lambda+\mu+\alpha} & 0 & \cdots & 0\\
	0 & \frac{\lambda(1-(x-\lfloor x \rfloor))}{\lambda+\mu+\alpha} & \cdots & 0 \\
	\vdots & \vdots & \ddots & \vdots  \\
	0 & 0 & \cdots &  \frac{\lambda(1-(x-\lfloor x \rfloor))}{\lambda+\mu+\alpha}
	\end{bmatrix} +
	\begin{bmatrix}
	0 & 0 & \cdots & \frac{\mu (1-q)}{\lambda+\mu+\alpha}\\
	\frac{\mu (1-q)}{\lambda+\mu+\alpha} & 0 & \cdots & 0 \\
	\vdots & \ddots & \ddots & \vdots  \\
	0 & \cdots & \frac{\mu (1-q)}{\lambda+\mu+\alpha} &  0
	\end{bmatrix}
	\in \mathbb{R}^{\lfloor x \rfloor \times \lfloor x \rfloor}	\, 
	\end{split} \\
	\begin{split}
	&A_{0}^{(k)} = \begin{bmatrix}
	\frac{\lambda}{\lambda+\mu+\alpha} & 0 & \cdots & 0\\
	0 & \frac{\lambda}{\lambda+\mu+\alpha} & \cdots & 0 \\
	\vdots & \vdots & \ddots & \vdots  \\
	0 & 0 & \cdots &  \frac{\lambda}{\lambda+\mu+\alpha}
	\end{bmatrix} +
	\begin{bmatrix}
	0 & 0 & \cdots & \frac{\mu (1-q)}{\lambda+\mu+\alpha}\\
	\frac{\mu (1-q)}{\lambda+\mu+\alpha} & 0 & \cdots & 0 \\
	\vdots & \ddots & \ddots & \vdots  \\
	0 & \cdots & \frac{\mu (1-q)}{\lambda+\mu+\alpha} &  0
	\end{bmatrix}
	\in \mathbb{R}^{k \times k} \qquad k = \lfloor x \rfloor+1
	\end{split}\\
	\begin{split}
	&A_{1}^{(k)} = 
	\begin{bmatrix}
	\frac{\lambda}{\lambda+\mu+\alpha} & 0 & \cdots & 0 & 0\\
	0 & \frac{\lambda}{\lambda+\mu+\alpha} & \cdots & 0 & 0\\
	\vdots & \vdots & \ddots & \vdots & \vdots \\
	0 & 0 & \cdots &  \frac{\lambda}{\lambda+\mu+\alpha} & 0
	\end{bmatrix}
	\in \mathbb{R}^{k \times (k+1)}
	\,  \quad k = 1, \cdots, \lfloor x \rfloor-1
	\end{split} 
		\end{align}
	\begin{align}
	\begin{split}
	&A_{1}^{(\lfloor x \rfloor)} = \begin{bmatrix}
	\frac{\lambda p}{\lambda+\mu+\alpha} & 0 & \cdots & 0 & 0\\
	0 & \frac{\lambda p}{\lambda+\mu+\alpha} & \cdots & 0 & 0\\
	\vdots & \vdots & \ddots & \vdots & \vdots \\
	0 & 0 & \cdots &  \frac{\lambda p}{\lambda+\mu+\alpha} & 0
	\end{bmatrix}
	\in \mathbb{R}^{\lfloor x \rfloor \times (\lfloor x \rfloor+1)} \qquad A_1^{(\lfloor x \rfloor)} = \bm{0}_{\lfloor x \rfloor \times (\lfloor x \rfloor+1)} \,
	\end{split}
	\end{align}

    \section{Algorithm 1} \label{appendix:algorithm1}
    \begin{algorithm}[h]
	\caption{}\label{alg:Poisson}
	\begin{algorithmic}[1]
		\Procedure{Calculate $U^{(j)}$, $\Gamma^{(j)}$, $G^{(j)}$}{} \Comment{The $U^{(j)}, \Gamma^{(j)}, G^{(j)}$ for $P_\alpha	^{(x)}$}
		\State $U^{(\lfloor x \rfloor+1)} \gets A_0^{(\lfloor x \rfloor+1)}$
		\State $\Gamma^{(\lfloor x \rfloor+1)} \gets A_1^{(\lfloor x \rfloor)} \, (\mathbf{I}- U^{(\lfloor x \rfloor+1)})^{-1}$
		\State $G^{(\lfloor x \rfloor+1)} \gets (\mathbf{I}- U^{(\lfloor x \rfloor+1)})^{-1} \, A_{-1}^{(\lfloor x \rfloor+1)}$
		\For {$j = \lfloor x \rfloor:2$}
		\State $U^{(j)} \gets A_0^{(j)} + A_1^{(j)}\, G^{(j+1)}$
		\State $\Gamma^{(j)} \gets A_1^{(j-1)} \, (\mathbf{I}- U^{(j)})^{-1}$
		\State $G^{(j)} \gets (\mathbf{I}- U^{(j)})^{-1} \, A_{-1}^{(j)}$
		\EndFor
		\State \textbf{end}
		\EndProcedure
		\Procedure{Poisson's Equation}{$U^{(j)}, \Gamma^{(j)}, G^{(j)}, \bm{g}_\alpha$}
		\For {$i = 1:\lfloor x \rfloor +1$}
		\State $\bm{g}_i = \bm{g}_\alpha (\frac{i(i-1)}{2}+1:\frac{i(i+1)}{2})$
		\EndFor
		\State \textbf{end}
		\State $y(1) \gets 0$
		\For {$j = 1:\lfloor x \rfloor+1$}
		\State $y(j) \gets (\mathbf{I} - U^{(j+1)})^{-1} \, (\bm{g}_{j}+ \sum_{k = j}^{\lfloor x \rfloor} \Pi_{l=j+1: k+1} \, \Gamma^{(l)} \, \bm{g}_{k+2}) + \, G^{(j+1)} \, y(j-1)$
		\EndFor
		\State \textbf{end}
		\State $y(1) \gets \frac{\mu q}{\lambda + \mu+\alpha} + A_1^{(1)}\, y(2)$
		\State $\bm{\gamma}^{(x)}_\alpha(1)= \frac{y(1)}{1-(A_0^{(1)} \, + \, A_1^{(1)} \, G^{(2)})}$ 
		\For {$j = 2: \lfloor x \rfloor+1$}
		\State $\bm{\gamma}^{(x)}_\alpha(\frac{j(j-1)}{2}+1: \frac{j(j-1)}{2}+j) = y(j) \, + \, \Pi_{l=j:2} \, G^{(l)} \, \bm{\gamma}^{(x)}_\alpha(1)$
		\EndFor
		\State \textbf{end}
		\State \textbf{return} $\bm{\gamma}^{(x)}_\alpha$
		\EndProcedure
	\end{algorithmic} 
\end{algorithm}

\section{Expression for the expected discount factor in the $R$-case}
\label{appendix:wijR}
{\scriptsize
	\begin{align}
	&\mathbb{E}\left(e^{-\alpha \widehat{W}^{(\lfloor x \rfloor+1, x)}_{i,j}}\right) \\
	=& \mathbb{E}\left( \mathbb{E}\left(e^{-\alpha \widehat{W}^{(\lfloor x \rfloor+1, x)}_{i,j}} \mid T(t), X(T(t))\right) \right) \notag \\
	=& \int_{0}^{\infty} (\lambda+\mu) \, e^{-(\lambda+\mu) y}  \, \sum_{(i', j')} P(X(T(t)) = (i',j') \mid X(t) = (i,j) ) \, \mathbb{E}\left(e^{-\alpha \widehat{W}^{(\lfloor x \rfloor+1, x)}_{i,j}} \mid T(t) = y, X(T(t)) = (i',j') \right)  dy \notag \\
	=& \int_{0}^{\infty} (\lambda+\mu) \, e^{-(\lambda+\mu) y}  \, \sum_{(i', j')} P(X(T(t)) = (i',j') \mid X(t) = (i,j) ) \, \mathbb{E}\left(e^{-\alpha (y+\widehat{W}^{(\lfloor x \rfloor+1, x)}_{i',j'})} \right)  dy \notag\\
	=& \left( \sum_{(i', j')} P(X(T(t)) = (i',j') \mid X(t) = (i,j) ) \, \mathbb{E} \left(e^{-\alpha \widehat{W}^{(\lfloor x \rfloor+1, x)}_{i',j'}} \right) \, \int_{0}^{\infty} (\lambda+\mu) \, e^{-(\lambda+\mu+\alpha) y}  \,  dy \right) \notag \\
	=& \frac{\lambda}{\lambda+\mu+\alpha} \left( \mathbb{E} \left(e^{-\alpha \widehat{W}^{(\lfloor x \rfloor+1, x)}_{i, j+1}}\right)\mathbbm{1}_{\{i < \lfloor x \rfloor\}}  + \left( p \,  \mathbb{E} \left(e^{-\alpha \widehat{W}^{(\lfloor x \rfloor+1, x)}_{i, j+1}}\right) + (1-p) \, \mathbb{E} \left(e^{-\alpha \widehat{W}^{(\lfloor x \rfloor+1, x)}_{i, j}}\right) \right) \mathbbm{1}_{\{i = \lfloor x \rfloor\}}  + \mathbb{E} \left(e^{-\alpha \widehat{W}^{(\lfloor x \rfloor+1, x)}_{i, j}}\right) \mathbbm{1}_{\{i > \lfloor x \rfloor\}}  \right) \notag \notag\\
	&   \quad+ \frac{\mu}{\lambda+\mu+\alpha} \, \left( q+(1-q)(1-p) \mathbb{E} \left(e^{-\alpha W^{(x)}_{i-1, j-1}}\right) \mathbbm{1}_{\{j = \lfloor x\rfloor +1\}} +(1-q)(1-(1-p)\mathbbm{1}_{\{j = \lfloor x\rfloor +1\}})\mathbb{E} \left(e^{-\alpha W^{(x)}_{i-1, j}}\right) \right. \notag \\
	&  \left. \quad + \left(q  + (1-q) \, \mathbb{E} \left(e^{-\alpha W^{(\lfloor x \rfloor+1, x)}_{j, j}}\right) \right) \mathbbm{1}_{\{i = 1\}}  \right)   \notag
	\end{align}
}
	
	\section{The reneging case}\label{appendix: 2.1}
	\begin{align}
	&\widehat{P}_\alpha^{(\lfloor x \rfloor+1,x)} = \begin{bmatrix}
	A_0^{(1)} & A_1^{(1)} & \cdots & \cdots & \cdots & \cdots & 0\\
	A_{-1}^{(2)} & A_0^{(2)} & A_1^{(2)}  & \cdots & \cdots & \cdots & 0\\
	\vdots & A_{-1}^{(3)} & A_0^{(3)} & A_1^{(3)}   & \cdots & \cdots & 0 \\
	\vdots & \vdots & A_{-1}^{(4)} & A_0^{(4)} & A_1^{(4)} & \cdots & 0 \\
	\vdots &\vdots & \vdots & \vdots & \vdots & \ddots & \vdots \\
	0 &0 &0 & 0 &  0 & \widehat{A}_{-1}^{(\lfloor x \rfloor+1)} & \widetilde{A}_0^{(\lfloor x \rfloor+1)} 
	\end{bmatrix} \,,\\
	&\widehat{A}_{-1}^{(\lfloor x \rfloor+1)} = \begin{bmatrix}
	0 & 0 & \cdots & 0\\
	\frac{\mu q + \mu (1-q)(1-(x-\lfloor x \rfloor))}{\lambda+\mu+\alpha} & 0 & \cdots & 0 \\
	0 & \frac{\mu q + \mu (1-q)(1-(x-\lfloor x \rfloor))}{\lambda+\mu+\alpha}  & \cdots & 0 \\
	\vdots & \vdots & \ddots & \vdots  \\
	0 & 0 & \cdots &  \frac{\mu q + \mu (1-q)(1-(x-\lfloor x \rfloor))}{\lambda+\mu+\alpha}
	\end{bmatrix} \in \mathbb{R}^{(\lfloor x \rfloor + 1) \times  \lfloor x \rfloor} \\
	& \widetilde{A}_0^{(\lfloor x \rfloor+1)} =  \begin{bmatrix}
	\frac{\lambda}{\lambda+\mu+\alpha} & 0 & \cdots & 0\\
	0 & \frac{\lambda}{\lambda+\mu+\alpha} & \cdots & 0 \\
	\vdots & \vdots & \ddots & \vdots  \\
	0 & 0 & \cdots &  \frac{\lambda}{\lambda+\mu+\alpha}
	\end{bmatrix} +
	\begin{bmatrix}
	0 & 0 & \cdots & \frac{\mu (1-q)}{\lambda+\mu+\alpha} 
	\\
	\frac{\mu (1-q)\, (x - \lfloor x \rfloor)}{\lambda+\mu+\alpha} & 0 & \cdots & 0 \\
	\vdots & \ddots & \ddots & \vdots  \\
	0 & \cdots & \frac{\mu (1-q) \, (x - \lfloor x \rfloor)}{\lambda+\mu+\alpha} &  0
	\end{bmatrix}
	\in \mathbb{R}^{(\lfloor x \rfloor+1)  \times (\lfloor x \rfloor+1 )} 
	\end{align}
\end{appendix}
\end{document}